\documentclass[journal]{IEEEtran}

\usepackage{cite}
\usepackage{color,soul}
\usepackage[hyphens]{url}
\usepackage{multirow}
\usepackage{graphicx}
\usepackage{enumerate}
\usepackage{enumitem}
\usepackage{amsmath}
\usepackage{amsthm}
\usepackage{multirow}
\usepackage{array}
\usepackage{booktabs} 
\usepackage{subcaption}
\usepackage{algorithm}
\usepackage[noend]{algpseudocode}
\usepackage{bm}
\usepackage{amssymb}

\DeclareMathOperator*{\argmin}{arg\,min}

\DeclareMathOperator*{\Minimize}{Minimize:}

\DeclareMathOperator*{\SubjectTo}{Subject\phantom{a}to:}
\DeclareMathOperator*{\supp}{\textsf{supp}}

\newcommand{\norm}[1]{\left\lVert#1\right\rVert} 
\newcommand{\abs}[1]{\left\lvert#1\right\rvert}

\newtheorem{lemma}{\bf{Lemma}}
\newtheorem{definition}{\bf{Definition}}
\newtheorem{proposition}{\bf{Proposition}}
\newtheorem{theorem}{\bf{Theorem}}
\newtheorem{corollary}{\bf{Corollary}}
\newtheorem{remark}{Remark}
\newcolumntype{L}{>{\centering\arraybackslash}m{4cm}}

\ifCLASSINFOpdf
\else
\fi

\begin{document}
%
\title{Resilient Optimal Estimation Using Measurement Prior}
%
%
%

\author{Olugbenga Moses Anubi,~\IEEEmembership{Member,~IEEE,}
        Charalambos Konstantinou,~\IEEEmembership{Member,~IEEE,}
        and~Rodney~Roberts,~\IEEEmembership{Senior Member,~IEEE}
\thanks{The authors are with the Department
of Electrical and Computer Engineering, FAMU-FSU College of Engineering, Tallhassee FL.}%
\thanks{Emails : oanubi@fsu.edu*, ckonstantinou@fsu.edu, rroberts@fsu.edu.}
}

%
%

\markboth{}%
{Shell \MakeLowercase{Anubi \textit{et al.}}: Resilient Optimal Estimation Using Measurement Prior}
%



\maketitle

\begin{abstract}
This paper considers the problem of optimal estimation for linear system with the measurement vector subject to arbitrary corruption by an adversarial agent. This problem is relevant to cyber-physical systems where, due to the tight coupling of physics, communication and computation, a malicious agent is able to exploit multiple inherent vulnerabilities in order to inject stealthy signals into the measurement process. These malicious signals are calculated to serve the attack objectives of causing false situation awareness and/or triggering a sequence of cascading effects leading to an ultimate system failure.  We assume that the attacker can only compromise a portion, but not all, of the measurement channels simultaneously. However, once a channel is compromised, the attacker is free to modify the corresponding measurement arbitrarily. 

Consequently, the problem is formulated as a compressive sensing problem with additional prior-information model. The prior-information considered is a set inclusion constraint on the measurement vector. It is shown that if the prior set satisfies certain conditions, the resulting recovery error bound is much stronger. The approach is applied to the problem of resilient sate estimation of a power system. For this application, Gaussian Process is used to build a prior generative probabilistic regression model from historical data. The resulting Gaussian Process Regression model recursively maps energy market information to \emph{iid} Gaussian distributions on the relevant system measurements. An optimization-based resilient state estimator is then developed using a re-weighted $\ell_1$-minimization scheme. The developed algorithm is evaluated through a numerical simulation example of the IEEE 14-bus system mapped to the New York Independent System Operator (NYISO) grid data.
\end{abstract}

\begin{IEEEkeywords}
Resilient estimation, Compressive Sensing, Auxiliary models.
\end{IEEEkeywords}

%
\IEEEpeerreviewmaketitle

\section{Notation}\label{s:notation}

The following notions and conventions are employed throughout the paper:
$\mathbb{N}$ denotes the set of natural numbers. $\mathbb{R},\mathbb{R}^m,\mathbb{R}^{m\times n}$  denote the space of real numbers, real vectors of length $m$ and real matrices of $m$ rows and $n$ columns respectively.
$\mathbb{R}_+$ denotes positive real numbers.
$X^\top$ denotes the transpose of the quantity $X$.
Normal-face lower-case letters ($x\in\mathbb{R}$) are used to represent real scalars, bold-face lower-case letter ($\mathbf{x}\in\mathbb{R}^m$) represents vectors, normal-face upper case ($X\in\mathbb{R}^{m\times n}$) represents matrices, while calligraphic upper case letters (e.g $\mathcal{T}$) represent sets. Let $\mathcal{T}\subseteq\{1,\hdots,m\}$ then, for a matrix $X\in\mathbb{R}^{m\times n}$, $X_\mathcal{T} \in\mathbb{R}^{\abs{\mathcal{T}}\times n}$ and $X^{\mathcal{T}} \in\mathbb{R}^{m\times \abs{\mathcal{T}}}$ are the sub-matrices obtained by extracting the rows, and columns respectively, of $X$ corresponding to the indices in $\mathcal{T}$. $\mathcal{N}(X)$, $\mathcal{R}(X)$ and $\overline{\sigma}(X)$ denote the null space, range space and the largest singular value of the matrix $X$ respectively. For a vector $\mathbf{x}$, $\mathbf{x}_i$ denotes its $i$th element. The support of a vector $\mathbf{x}\in\mathbb{R}^m$ is denoted by $\supp(\mathbf{x})\triangleq\left\{i\hspace{1mm}:\hspace{1mm}\mathbf{x}_i\neq0\right\}$, with $\abs{\supp(\mathbf{x})}\le m$ being the number of nonzero elements of $\mathbf{x}$. $\mathcal{S}_k^m\triangleq\left\{\mathbf{x}\in\mathbb{R}^m\setminus \{0\}\hspace{1mm}:\hspace{1mm}\abs{\supp(\mathbf{x})}\le k\right\}$ denotes the set of all nonzero $k$-sparse vectors. The superscript $m$ is dropped whenever the dimension is clear from context. The $p$-norm of a vector $\mathbf{x}\in\mathbb{R}^m$ is defined as $\norm{\mathbf{x}}_p\triangleq\left(\sum\limits_{i=1}^m\abs{\mathbf{x}_i}^p\right)^{\frac{1}{p}}$. Given a vector $\mathbf{x}\in\mathbb{R}^m$, the following inequality about vector norms 
\begin{align*}
    \norm{\mathbf{x}}_q\le\norm{\mathbf{x}}_p\le m^{\left(\frac{1}{p}-\frac{1}{q}\right)}\norm{\mathbf{x}}_q,\hspace{2mm} 0<p\le q\le\infty
\end{align*}
is useful for some results down the line.
Given a positive scalar $\delta\in\mathbb{R}_+$, a saturation function $\textsf{sat}_\delta:\mathbb{R}\mapsto[-\delta,\hspace{1mm}\delta]$ is given by
\begin{align*}
    \textsf{sat}_\delta(x) = \left\{\begin{array}{rcl}-\delta&\text{if}&x<-\delta\\x&\text{if}&\abs{x}\le\delta\\\delta&\text{if}&x>\delta\end{array}\right.
\end{align*}
A best $k$th term approximation of a vector $\mathbf{e}\in\mathbb{R}^m$ is denoted by $\mathbf{e}[k] \triangleq\min\limits_{\norm{\mathbf{f}}_0=k}\norm{\mathbf{e}-\mathbf{f}}_1$ .
\section{Introduction}

Cyber-physical systems (CPS) refer to a generation of systems with tightly-integrated communication, computational and physical capabilities that can interact with humans through many new modalities \cite{gill2008vision,baheti2011cyber}. Such systems are fundamental to the operation of various safety-critical applications (e.g smart grid, connected \& autonomous vehicles (CAV), etc). Their failure can cause irreversible damage to the underlying physical system as well as to the humans who operate it or depend on it. For example, critical infrastructure domains are composed of a multitude of CPS of various scales and at all levels. The control of CPS is enabled by the proliferation of sensing devices which allow geographically isolated physical plants to be remotely monitored. Field embedded devices, typically called remote terminal units (RTUs), deployed in large-scale, geographically-sparse CPS collect measurements related to the physical process. The measured data are sent via supervisory control and data acquisition (SCADA) systems to central master stations. At the central site, the information from RTUs is utilized to carry out necessary analysis and control, e.g., determine if a leak has occurred and the level of criticality. A critical function at the management system level is to estimate the state variables of the CPS. These state estimates are then used to adjust the control of the physical space. In power systems, for instance, once the operating state is known, estimates are utilized for energy management system application functions such as optimal flow control, automatic generation control, and contingency analysis. The results of such functions are used in order to take preventive and corrective actions as well as ensure secure and reliable operation of the CPS. Due to the significance of state estimation routines, it is of paramount importance that such algorithms incorporate proper mechanisms for operating resiliently in the event of malicious events \cite{mclaughlin2016cybersecurity}.

Sophisticated attackers who are able to gain unauthorized access to the communication network of a CPS can modify the transmitted measurements to the central control and estimation stations \cite{liu2011false}, thereby causing a false situation awareness or triggering a cascade of events ultimately leading to a system failure. Furthermore, adversaries can hack into the RTUs or even infiltrate secondary channels of the supply chain in order to distort the measurements \cite{konstantinou2016case}. Existing work on the topic has shown that this class of \emph{false data injection attacks (FDIAs)} can bypass \emph{bad data detection (BDD)} schemes and inject errors in the resulting state estimation without being detected \cite{liu2011false, liang2017review, deng2017false, liang20172015}. Such detection methods are residual schemes traditionally based on the largest normalized residual between the obtained measurements and the predicted values from the system estimated states \cite{wu2018bad}. The impact of FDIAs, on power systems for instance, could skew the electricity markets in favor of the attacker or even result in masking the outage of lines and removing the attacked RTUs from the network \cite{liu2015impacts, kosut2010limiting}. Existing work on addressing the vulnerability of FDIAs typically rely on protecting a set of devices (and thus a set of measurements) or verifying each state variable independently. The high computational and deployment cost, as well as the associated risks of these methods, have hampered their feasibility for use in practical real-time systems \cite{liang2017review}. Moreover, estimation techniques developed for specific system configurations \cite{ashok2018online} often exhibit poor resiliency performance, in general, against FDIAs. Therefore, more computationally feasible, adaptive, and real-time implementable resiliency strategies are needed. The design of such estimators need to consider adverse settings in order to reliably estimate CPS state variables. 

Consequently, the attack-resilient state estimation has attracted significant attention in recent literature \cite{cardenas2008research}. While there are numerous work on resilient state estimation, we focus on the ones that are optimization-based -- since our work ultimately depends on solving a convex program. One of the earliest work employing optimization \cite{fawzi2014secure} formulated the resilient estimation problem for an LTI system as a compressive sensing (CS) problem and used standard results\cite{candes2005decoding} from the CS community to create a convex relaxation of the resulting optimization problem. Following that, a number of papers have either modified or extended the framework to include measurement noise\cite{pajic2015attack,yong2015resilient}, time varying attack support \cite{hu2016secure}, robustness considerations \cite{pajic2014robustness} and distributed case \cite{kekatos2012distributed}. There are also numerous applications including but not limited to; power systems \cite{deng2017false}, UAVS \cite{fiore2017secure,hu2016secure}, energy delivery systems\cite{mestha2017cyber}, autonomous vehicles  and networked systems.

In this paper, we build on our previous works on enhancing the \emph{recoverability} of resilient estimators by incorporating prior information, either in form of attack-support estimation\cite{anubi2018robust} or through a more general set inclusion constraint\cite{anubi2019enhanced}. Here, we provide theoretical guarantees of how certain \emph{boundedness} property of the prior information set can improve the reconstruction error bound of the resulting resilient estimator. Unlike the previous work \cite{fawzi2014secure,hu2016secure,chang2018secure} which depend on the Restricted Isometry Property (RIP) \cite{candes2005decoding}, we have derived our results using a related Nullspae Property (NSP) \cite{cohen2009compressed}. The reason for this is given in subsequent sections. Moreover, a numerical example is given in which the developed estimator is applied to the NYISO transmission grid. The prior information generates a likelihood-level ellipsoid constraints on the ``true" measurement vector via a Gaussian Process Regression (GPR) mean and covariance functions of the locational marginal bus prices. This example demonstrates tremendous improvement in resiliency by using readily available auxiliary measurements to corroborate the state estimation process using the proposed scheme.

The remaining of the paper is organized as follows: in Section \ref{s:background} we provide necessary definitions and background for this work. Section \ref{s:methodology} presents the formulation of the estimation problem as well as our proposed solution algorithm for the enhanced state estimator. Experimental details and simulation results are described in Section \ref{s:experiments}. Our concluding remarks are discussed in Section \ref{s:conclusions}.

\section{Background}\label{s:background}
Consider a linear measurement model of the form:

\begin{align}\label{eqn:meas_model}
    \mathbf{y} = H\mathbf{x} + \mathbf{e},
\end{align}
where $H\in\mathbb{R}^{m\times n}$ is a measurement/coding matrix ($m>n$) and $\mathbf{y}\in\mathbb{R}^m$ is a measurement vector corrupted by an arbitrary unknown but sparse error vector $\mathbf{e}\in\mathbb{R}^m$. By sparsity, we mean that $\norm{\mathbf{e}}_0\le q$, for a given $q\le m$. In classical error correction problem \cite{berrou1993near,elias1954error}, the objective is to recover the input vector $\mathbf{x}\in\mathbb{R}^n$, given the corrupt measurement $\mathbf{y}$ and the matrix $H\in\mathbb{R}^{m\times n}$. Consequently an optimal decoder $\mathcal{D}_0:\mathbb{R}^m\mapsto\mathbb{R}^n$ is considered, of the form:
\begin{align}\label{eqn:opt_dec0_1}
    \mathcal{D}_0(\mathbf{y}) = \argmin\limits_{\mathbf{x}\in\mathbb{R}^n}{\norm{\mathbf{y}-H\mathbf{x}}_0}.
\end{align}
Evidently, the existence of such decoder is equivalent to the uniqueness of the underlying index minimization problem.

Suppose, the coding matrix $H$ is full rank. Let 
\begin{align}
    H = QR = \left[\begin{array}{cc}Q_1&Q_2\end{array}\right]\left[\begin{array}{c}R_1\\0\end{array}\right],
\end{align}
be the QR decomposition of $H$, where $Q\in\mathbb{R}^{m\times m}$ is orthogonal, $Q_1\in\mathbb{R}^{m\times n}$, $Q_2\in\mathbb{R}^{m\times(m-n)}$, and $R_1\in\mathbb{R}^{n\times n}$ is a full rank upper triangular matrix. Multiplying the left and right hand sides of \eqref{eqn:meas_model} by $Q_2^\top$, the transformed measurement model becomes:
\begin{align}
    Q_2^\top\mathbf{y} = Q_2^\top\mathbf{e}.
\end{align}
Thus, the optimal decoder $\mathcal{D}_0:\mathbb{R}^m\mapsto\mathbb{R}^n$ is given by
\begin{align}\label{eqn:opt_dec0_2}
    \mathcal{D}_0(\mathbf{y}) = R_1^{-1}Q_1^\top\left(\mathbf{y} - \argmin\limits_{Q_2^\top\left(\mathbf{y}-\mathbf{e}\right)=0}{\norm{e}_0}\right),
\end{align}
which is equivalently related with the compressive sensing problem\cite{candes2005decoding}:
\begin{align}\label{eqn:comp_sens}
\Minimize\limits_{\mathbf{e}}{\left\|\mathbf{e}\right\|_0}\hspace{2mm}\SubjectTo\hspace{2mm}Q_2^\top(\mathbf{y}-\mathbf{e})=0.
\end{align}
Subsequently, we will consider the compressive sensing problem of the form in \eqref{eqn:comp_sens} for analysis purposes, and restrict ourselves to the decoder of the form in \eqref{eqn:opt_dec0_1}( or \eqref{eqn:opt_dec0_2}) for algorithm development. 

The obvious question that arises, then is to determine if there is a unique minimizer of the above index-minimizing optimization problem. The following proposition, adapted from \cite{hayden2016sparse}, gives the condition for the existence of a unique solution to the optimization problem in \eqref{eqn:comp_sens}.

\begin{proposition}[Uniqueness]\label{thm:prop_uniquess}
Given $k\in\mathbb{N}$, if every $2k$ columns of $Q_2^\top$ are linearly independent and there exists at least one $p\le k$ for which $\mathcal{S}_p\cap\left(\mathcal{N}(Q_2^\top)+\mathbf{y}\right)\neq \varnothing$, then the optimization problem in \eqref{eqn:comp_sens} has a unique solution. 
\end{proposition}
\begin{proof}
It suffices to show that, for all $p\le k$, the feasible region $\mathcal{R}_p\triangleq \left\{\mathbf{e}\in\mathbb{R}^m|\norm{\mathbf{e}}_0=p, Q_2^\top\left(\mathbf{e}-\mathbf{y}\right)=0\right\}=\mathcal{S}_p\cap\left(\mathcal{N}(Q_2^\top)+\mathbf{y}\right)$ is a singleton. If this is true, then the result follows from the existence of at least one feasible point for some $p\le k$. To see that $\mathcal{R}_p$ is a singleton, let $\mathbf{e}_1,\mathbf{e}_2\in\mathcal{R}_p$, $\mathbf{e}_1\neq\mathbf{e}_2$, then $Q_2^\top\left(\mathbf{e}_1-\mathbf{e}_2\right)=0$. Since every $2s$ columns of $Q_2^\top$ are linearly independent, then the last equation is true iff $\norm{\mathbf{e}_1-\mathbf{e}_2}_0>2s\Rightarrow \norm{\mathbf{e}_1}_0+\norm{\mathbf{e}_2}_0>2k\Rightarrow p>k$, a contradiction. Thus, $\mathbf{e}_1=\mathbf{e}_2$, implying that $\abs{\mathcal{R}_p}=1 \hspace{2mm} \forall\hspace{2mm} p\le k$
\end{proof}

\begin{corollary} If there exists $p \le m$ such that $\mathcal{S}_{2p}\cap\mathcal{N}(Q_2^\top)=\varnothing$ and $\mathcal{S}_p\cap\left(\mathcal{N}(Q_2^\top)+\mathbf{y}\right)\neq \varnothing$, then the optimization problem in \eqref{eqn:comp_sens} has a unique solution. 
\end{corollary}
\begin{proof}
The statement ``every $2s$ columns of $Q_2^\top$ are linearly independent" implies that $\mathcal{S}_{2p}\cap\mathcal{N}(Q_2^\top)=\varnothing$ for $p\le k$. Thus the result follows from \eqref{thm:prop_uniquess}.
\end{proof}

The optimization problem in \eqref{eqn:comp_sens}, in most instances, does not lend itself to a solution in polynomial time due to the nonconvexity associated with the index-minimization objective. As a result, it is often replaced with its convex neighbor:
\begin{align}\label{eqn:comp_sens_L1}
\Minimize\limits_{\mathbf{e}}{\left\|\mathbf{e}\right\|_1\hspace{2mm}\SubjectTo\hspace{2mm}Q_2^\top\left(\mathbf{y}-\mathbf{e}\right)=0}.
\end{align}
As a result, naturally, questions arise about how well the this convex relaxation recovers the solution to the original problem, assuming a unique solution exists? For instance, under what condition(s) will the solution of \eqref{eqn:comp_sens_L1} recover the solution of the original problem \eqref{eqn:comp_sens}. This property called \emph{recoverability} has been studied extensively in compressive sensing literature, largely under the umbrella of either the so called \emph{Restricted Isometry Property} (RIP) or the \emph{Null Space Property} (NSP). While other notions have emerged in recent years, the RIP and NSP are the two most common conditions that one imposes on $Q_2^\top$ in order to guarantee recoverability. In what follows, we outline some RIP and NSP-based results that are relevant to this work.

\subsection{RIP-based results}
The RIP was introduced in \cite{candes2005decoding} to establish stable recoverability for the relaxed problem in \eqref{eqn:comp_sens_L1}. Ever since, there have been so many other follow-up results and refinements to the original guarantees published by Candes et. al. In what follow, we provide a tiny portion of existing results, slightly modified or built upon in some cases, that are relevant to this work. 
\begin{definition}[RIP \cite{candes2005decoding}]
A matrix $A$ has the RIP of sparsity $k$ if there exists $0<\delta<1$ such that
\begin{align}
    \left(1-\delta\right)\norm{\mathbf{x}}_2^2\le\norm{A\mathbf{x}}_2^2\le\left(1+\delta\right)\norm{\mathbf{x}}_2^2
\end{align}
for all $\mathbf{x}\in\mathcal{S}_k$. Moreover, the smallest $\delta$ for which the above inequality holds is called the \emph{restricted isometry constant}, and denoted as $\delta_k(A)$.
\end{definition}
The above definition essentially requires that every set of columns with cardinality less that or equal to $k$ behaves like an orthonormal system. The following theorem lists the recovery error due to relaxed convex program above.

\begin{theorem}[\cite{candes2005decoding},\cite{cai2013sparse}]
Let $\mathbf{e}$ be a sparse vector satisfying $Q_2^\top\left(\mathbf{y}-\mathbf{e}\right)=0$ and $\hat{\mathbf{e}}$ be the solution of \eqref{eqn:comp_sens_L1}. If $\displaystyle \delta_{2k}(Q_2^\top)<\frac{1}{\sqrt{2}}$, then 
\begin{align}
    \norm{\hat{\mathbf{e}}-\mathbf{e}}_2\le \frac{2}{\sqrt{k}}\left(\frac{\delta_{2k}+\sqrt{\delta_{2k}\left(\frac{1}{\sqrt{2}}-\delta_{2k}\right)}}{\sqrt{2}\left(\frac{1}{\sqrt{2}}-\delta_{2k}\right)}+1\right)\norm{\mathbf{e}-\mathbf{e}[k]}_1,
\end{align}
where $\mathbf{e}[k]$ is the best $k$-term approximation of $\mathbf{e}$.
\end{theorem}
\begin{remark}
If $\mathbf{e}\in\mathcal{S}_k$, then $\hat{\mathbf{e}}=\mathbf{e}$. Thus, if $\displaystyle \delta_{2k}(Q_2^\top)<\frac{1}{\sqrt{2}}$ the relaxed program in \eqref{eqn:comp_sens_L1} will recover any $k$-sparse vector $\mathbf{e}\in\mathcal{S}_k$ exactly! 
\end{remark}
\begin{remark}
While the RIP provides very nice theoretical guarantees, computing/numerically verifying the restricted isometry constant is NP-hard. However, for a large class of matrices, the RIP condition holds with overwhelming probability \cite{candes2006stable}.
\end{remark}
For any invertible matrix $U$, the matrix $UA$ share the same nullspace as $A$ but can have dramatically different RIP constants. This, at a first glance, might seem like a major drawback of RIP-based analyses, because the equivalent programs $\left\{\Minimize\limits_{\mathbf{x}}{\left\|\mathbf{x}\right\|_1}\hspace{2mm}\SubjectTo\hspace{2mm}A\mathbf{x}=\mathbf{b}\right\}$ and  $\left\{\Minimize\limits_{\mathbf{x}}{\left\|\mathbf{x}\right\|_1}\hspace{2mm}\SubjectTo\hspace{2mm}UA\mathbf{x}=U\mathbf{b}\right\}$ may end up having totally different RIP-based recoverability properties. To overcome this situation, many researchers have derived their results using subspace-based analysis, which generally \emph{mods} out such transformations and provide a more uniform result. Next, we examine the \emph{nullspace} property, which has been widely used for such purpose.

\subsection{NSP-based results}
The term \emph{nullspace property} originates from \cite{cohen2009compressed}. It gives necessary and sufficient conditions for recoverability. Like RIP, numerical verification of the NSP is combinatorial and NP-hard.

\begin{definition}[$\textsf{NSP}_q$,\cite{chen2012stability}]
A matrix $A$ is said to satisfy the nullspace property with parameters $\gamma\in\mathbb{R}_+$ and $k\in\mathbb{N}$, denoted by $A\in\textsf{NSP}_q(k,\gamma)$, if every nonzero $\mathbf{e}\in\mathcal{N}(A)$ satisfies
\begin{align*}
    \norm{\mathbf{e}_\mathcal{T}}_q<\gamma\norm{\mathbf{e}_{\mathcal{T}^c}}_q
\end{align*}
for all $\mathcal{T}\subset \left\{\begin{array}{ccc}1&\hdots&n\end{array}\right\}$ with $\abs{\mathcal{T}}\le k$.
\end{definition}
The following results list some recoverability results based on the NSP.
\begin{theorem}[\cite{donoho2001uncertainty,gribonval2003sparse}]\label{thm:eqv_prog}
The convex program in \eqref{eqn:comp_sens_L1} uniquely recovers all $k$-sparse vector $\mathbf{e}\in\mathcal{S}_k$ if and only if $Q_2^\top\in\textsf{NSP}_1(k,1)$
\end{theorem}

\begin{theorem}\label{thm:NSP}
Let $\mathbf{e}\in\mathbb{R}^m$ be a vector satisfying $Q_2^\top\left(\mathbf{y}-\mathbf{e}\right)=0$ and $\hat{\mathbf{e}}$ be the solution of \eqref{eqn:comp_sens_L1}. If $Q_2^\top\in\textsf{NSP}_q(k,\gamma)$ for some $0<\gamma<1$ and $q>1$, then 
\begin{align}
    \norm{\hat{\mathbf{e}}-\mathbf{e}}_1\le \frac{m}{\sqrt{2}}\left(\frac{4\left(1+\gamma\right)}{m\left(1-\gamma\right)}\right)^{\frac{1}{q}}\norm{\mathbf{e}-\mathbf{e}[k]}_1,
\end{align}
where $\mathbf{e}[k]$ is a best $k$-term approximation of $\mathbf{e}$.
\end{theorem}
\begin{proof}
From the results in \cite{chen2012stability}(Theorem III.4.1), the following inequality holds:
\begin{align*}
    \norm{\hat{\mathbf{e}}-\mathbf{e}}_q\le \frac{1}{\sqrt{2}}\left(\frac{4\left(1+\gamma\right)}{\left(1-\gamma\right)}\right)^{\frac{1}{q}}\norm{\mathbf{e}-\mathbf{e}[k]}_q.
\end{align*}
The result follows by using the following well-known norm inequality for $q>1$:
\begin{align*}
\norm{x}_q\le\norm{\mathbf{x}}_1\le m^{1-\frac{1}{q}}\norm{\mathbf{x}}_q.
\end{align*}
\end{proof}

\begin{remark}
This result demonstrates how the choice of $q$ in the parameterized \emph{nullspace property} $\textsf{NSP}_q$ can be used to modify the error bound. It is also worth noting that the $\textsf{NSP}_q$ may be quite different for different $q$-s. A nice entity relationship diagram for $\textsf{RIP}$, $\textsf{NSP}$ and coherence is also given in Figure III.2 of \cite{chen2012stability}. It would be nice to see the resulting error bounds change with these quantities laid out on the same diagram, although not pursued for this paper.
\end{remark}
\begin{remark}
It is noteworthy that as $q\rightarrow\infty$, the upper bound in Theorem~\ref{thm:NSP} approaches the uniform bound
\begin{align}
    \norm{\hat{\mathbf{e}}-\mathbf{e}}_1\le \frac{m}{\sqrt{2}}\norm{\mathbf{e}-\mathbf{e}[k]}_1.
\end{align}
\end{remark}
\begin{theorem}[maximum correctable errors]
Suppose that the nonzero vector $\mathbf{e}\in\mathbb{R}^m$ satisfies
\begin{align*}
		\norm{\mathbf{e}_\mathcal{T}}_q<\gamma\norm{\mathbf{e}_{\mathcal{T}^c}}_q,\hspace{2mm}\gamma\in(0,1), q>1
\end{align*}
for all $\mathcal{T}\subset\left\{1,\hdots,m\right\}$ satisfying $\abs{\mathcal{T}}\le k$. Then
\begin{align}
k<\frac{\gamma^q}{1+\gamma^q}m
\end{align}
\end{theorem}
\begin{proof}
Suppose, without loss of generality, that $\abs{\mathbf{e}_1}\ge\hdots\ge\abs{\mathbf{e}_m}$. Then,
\begin{align*}
		\sum\limits_{i=1}^{\abs{\mathcal{T}}}{\abs{\mathbf{e}_i}^q}<\gamma^q\sum\limits_{i=\abs{\mathcal{T}}+1}^{m}{\abs{\mathbf{e}_i}^q}.
\end{align*}
Observe that $\abs{\mathbf{e}_{\abs{\mathcal{T}}}}$, otherwise the right hand side of the above inequality would be zero identically and the strict inequality in the hypothesis could not hold. Next, dividing through by $\abs{\mathbf{e}_{\abs{\mathcal{T}}}}^q$ and observing that
\begin{align*}
\frac{\abs{\mathbf{e}_i}}{\abs{\mathbf{e}_{\abs{\mathcal{T}}}}}\left\{\begin{array}{lc}\ge1&\text{ if }i\le\abs{\mathcal{T}}\\\\\le1&\text{ if }i>\abs{\mathcal{T}}\end{array}\right..
\end{align*}
Thus,
\begin{align*}
		\abs{\mathcal{T}}\le\sum\limits_{i=1}^{\abs{\mathcal{T}}}{\left(\frac{\abs{\mathbf{e}_i}}{\abs{\mathbf{e}_{\abs{\mathcal{T}}}}}\right)^q}<\gamma^q\sum\limits_{i=\abs{\mathcal{T}}+1}^{m}{\left(\frac{\abs{\mathbf{e}_i}}{\abs{\mathbf{e}_{\abs{\mathcal{T}}}}}\right)^q}\le\gamma^q\left(m-\abs{\mathcal{T}}\right).
\end{align*}
Rearranging the terms of $\abs{\mathcal{T}}<\gamma^q\left(m-\abs{\mathcal{T}}\right)$ gives 
\begin{align*}
\abs{\mathcal{T}}<\frac{\gamma^q}{1+\gamma^q}m,
\end{align*}
which gives the desired result for all $\abs{\mathcal{T}}\le k$.
\end{proof}
\begin{remark}
For a given $k$, the result also gives a lower bound on admissible $\gamma$ as 
\begin{align*}
		\gamma>\left(\frac{k}{m-k}\right)^{\frac{1}{q}}.
\end{align*}
\end{remark}
The next result gives numerical sufficient conditions for $Q_2^\top\in\textsf{NSP}_1(k,1)$
\begin{theorem}
Given the unitary matrix $Q\in\mathbb{R}^{m\times m}$
\begin{align*}
Q = \left[\begin{array}{cc}Q_1&Q_2\end{array}\right],
\end{align*}
where $Q_1\in\mathbb{R}^{m\times n}$ and $Q_2\in\mathbb{R}^{m\times(m-n)}$, $n<m$ are orthogonal complements. For any integers $k<\frac{m}{2}$ and $q\ge2$, if
\begin{align}
		\norm{{Q_1}_\mathcal{T}}_q\triangleq\sup\limits_{\mathbf{x}\neq0}{\frac{\norm{{Q_1}_\mathcal{T}\mathbf{x}}_q}{\norm{\mathbf{x}}_q}}<\frac{1}{2}k^{\frac{1}{q}-1},
\end{align}
 for all $\mathcal{T}\subset\{1,2,\hdots,m\}$ with $\abs{\mathcal{T}}\le k$, then $Q_2^\top\in\textsf{NSP}_1(k,1)$. 
\end{theorem}

\begin{proof}
First, we observe that the inequality $\norm{\mathbf{e}}_q\le\norm{\mathbf{e}}_2\le\norm{\mathbf{e}}_1\le m^{1-\frac{1}{q}}\norm{\mathbf{e}}_q$ holds for all vector $\mathbf{e}\in\mathbb{R}^m$ and integer $q\ge2$. Thus,  for all $\mathcal{T}\subset\{1,2,\hdots,m\}$ with $\abs{\mathcal{T}}\le k$ and $\mathbf{x}\in\mathbb{R}^n$,
\begin{align*}
		\norm{{Q_1}_\mathcal{T}}_q &< \frac{1}{2}k^{\frac{1}{q}-1}\Rightarrow 2k^{1-\frac{1}{q}}\norm{{Q_1}_\mathcal{T}\mathbf{x}}_q<\norm{\mathbf{x}}_q\\
		&\Rightarrow 2\abs{\mathcal{T}}^{1-\frac{1}{q}}\norm{{Q_1}_\mathcal{T}\mathbf{x}}_q<\norm{\mathbf{x}}_q\\
		&\Rightarrow 2\norm{{Q_1}_\mathcal{T}\mathbf{x}}_1 < \norm{\mathbf{x}}_2 = \norm{Q_1\mathbf{x}}_2\\
		&\Rightarrow 2\norm{{Q_1}_\mathcal{T}\mathbf{x}}_1 < \norm{Q_1\mathbf{x}}_2 < \norm{Q_1\mathbf{x}}_1\\
		&\Rightarrow 2\norm{{Q_1}_\mathcal{T}\mathbf{x}}_1 < \norm{Q_1\mathbf{x}}_1 = \norm{{Q_1}_\mathcal{T}\mathbf{x}}_1 + \norm{{Q_1}_{\mathcal{T}^c}\mathbf{x}}_1\\
		&\Rightarrow \norm{{Q_1}_\mathcal{T}\mathbf{x}}_1 < \norm{{Q_1}_{\mathcal{T}^c}\mathbf{x}}_1\\
		&\Rightarrow Q_2^\top\in\textsf{NSP}_1(k,1)
\end{align*}
\end{proof}
\begin{remark}
		For $q=2$, the sufficient condition becomes $\overline{\sigma}({Q_1}_\mathcal{T})<\frac{1}{2\sqrt{k}}$ which imposes a limit on the amount of information any $k$-group of rows can convey of the orthogonal matrix $Q_1$. In other words, this ensures there is sufficient redundancy such if any $k$ combination of rows are deleted, the resulting system can still be used to reconstruct the state. This property is the motivation for the support refinement and row deletion scheme in \cite{anubi2018robust}.
\end{remark}
The following corollary gives a more specialized result based on $q=1$.
\begin{corollary}
Let $\mathbf{v}\in\mathbb{R}^m$ be a vector whose elements are the $\infty$-norm of the corresponding row of $Q_1$ i.e, $\mathbf{v}_i = \max\limits_{1\le j\le m}{\abs{{Q_1}_{ij}}}$.
If
\begin{align}
		\norm{\mathbf{v}[k]}_1<\frac{1}{2\sqrt{n}},
\end{align}
then $Q_2^\top\in\textsf{NSP}_1(k,1)$.
\end{corollary}
\begin{proof}
		First, we make the following observations for all $\mathbf{x}\in\mathbb{R}^n$
		\begin{align*}
				&\bullet\hspace{2mm}\norm{{Q_1}_\mathcal{T}\mathbf{x}}_1\le\left(\max\limits_{1\le j\le n}\left\{\norm{{Q_1}_{\mathcal{T}}^j}_1\right\}\right)\norm{\mathbf{x}}_1\\
				&\phantom{\bullet\hspace{2mm}\norm{{Q_1}_\mathcal{T}\mathbf{x}}_1}\le\norm{\mathbf{v}[k]}_1\norm{\mathbf{x}}_1.\\
				&\bullet\hspace{2mm}\norm{{Q_1}\mathbf{x}}_1\ge\norm{{Q_1}\mathbf{x}}_2=\norm{\mathbf{x}}_2\ge\frac{1}{\sqrt{n}}\norm{\mathbf{x}}_1\\
				&\phantom{\bullet\hspace{2mm}\norm{{Q_1}\mathbf{x}}_1}\Rightarrow \frac{1}{\sqrt{n}}\norm{\mathbf{x}}_1\le\norm{{Q_1}\mathbf{x}}_1.
		\end{align*}
		Thus, if $2\norm{\mathbf{v}[k]}_1<\frac{1}{\sqrt{n}}$, then 
		\begin{align*}
				&2\norm{{Q_1}_\mathcal{T}\mathbf{x}}_1\le2\norm{\mathbf{v}[k]}_1\norm{\mathbf{x}}_1<\frac{1}{\sqrt{n}}\norm{\mathbf{x}}_1\le\norm{{Q_1}\mathbf{x}}_1,\\
				&\Rightarrow 2\norm{{Q_1}_\mathcal{T}\mathbf{x}}_1 < \norm{Q_1\mathbf{x}}_1 = \norm{{Q_1}_\mathcal{T}\mathbf{x}}_1 + \norm{{Q_1}_{\mathcal{T}^c}\mathbf{x}}_1\\
				&\Rightarrow \norm{{Q_1}_\mathcal{T}\mathbf{x}}_1 < \norm{{Q_1}_{\mathcal{T}^c}\mathbf{x}}_1\\
				&\Rightarrow Q_2^\top\in\textsf{NSP}_1(k,1)
		\end{align*}
\end{proof}
\section{Resilient Estimation with Prior Information}\label{s:methodology}
Using prior information to enhance the recovery of sparse signals in compressive sensing is not a new idea \cite{friedlander2011recovering,anubi2018robust,miosso2009compressive,scarlett2012compressed}. However, vast majority of the existing literature focuses on prior information relating to the support of the sparse signal. In this paper, we consider prior information as a probability distribution over the system measurements. For cyber-physical systems, which are the primary subject of this study, such information is readily available via data-driven auxiliary models. 
In the light of model \eqref{eqn:meas_model} and the optimization problem in \eqref{eqn:comp_sens_L1}, consider the following slightly more general problem:
\begin{align}\label{eqn:comp_sens_subspace}
\Minimize\limits_{\mathbf{e}}{\left\|\mathbf{e}\right\|_1\hspace{2mm}\SubjectTo\hspace{2mm}\mathbf{y}-\mathbf{e}\in\mathcal{V}\cap\mathcal{X}},
\end{align}
where $\mathcal{V}\subset\mathbb{R}^m$ is a linear subspace satisfying the \emph{subspace property} $\norm{\mathbf{v}_\mathcal{T}}_1\le\gamma\norm{\mathbf{v}_{\mathcal{T}^c}}_1$, $\forall \mathbf{v}\in\mathcal{V}, \abs{\mathcal{T}}\le k<m$, and $\mathcal{X}\subset\mathbb{R}^m$ is a convex set with the bounded property $\norm{\mathbf{x}}_1\le\delta$, $\forall \mathbf{x}\in\mathcal{X}$. The bounded set adds extra layer of prior information which, as we will show next, improves the reconstruction error bound. While we have used a very simple bound here, other relevant property may be used to encode specialized prior information which can then lead to specialized result for the particular application. For instance; the bound could be probabilistic -- determined from the ROC characteristic of a data-driven, encode domain-specific relationship among the measurement channels.\\

\noindent We now have all the ingredients to state our main results:
\begin{theorem}\label{thm:main_result}
Consider the recovery optimization problem in \eqref{eqn:comp_sens_subspace}, where the linear subspace $\mathcal{V}$ satisfies the \emph{subspace property} $\norm{\mathbf{v}_\mathcal{T}}_1\le\gamma\norm{\mathbf{v}_{\mathcal{T}^c}}_1$, $\forall \mathbf{v}\in\mathcal{V}, \abs{\mathcal{T}}\le k<m$, and the convex set $\mathcal{X}\subset\mathbb{R}^m$ satisfies the bounded property $\norm{\mathbf{x}}_1\le\delta$, $\forall \mathbf{x}\in\mathcal{X}$. The reconstruction error with respect to any feasible vector $\mathbf{e}\in\mathbb{R}^m$ is bounded as:
\begin{align}
    \norm{\hat{\mathbf{e}}-\mathbf{e}}_1\le2\textsf{sat}_\delta\left(\frac{1+\gamma}{1-\gamma}\norm{\mathbf{e}-\mathbf{e}[k]}_1\right),
\end{align}
where $\mathbf{e}[k]$ is the best $k$-term approximation of $\mathbf{e}$.
\end{theorem}
\begin{remark}
This result is similar to existing recovery error-bound in literature \cite{chen2012stability}. The main difference lie in the saturation given by the bound on the prior-information set. This bound show up explicitly because of the way it was defined in the set. In some practical situation, such explicit bound may not exist. It is easy to modify the result based on the new characteristic of the prior-information set. In situations where the actual vector is only known to belong to the set $\mathcal{X}$ with some probability, the inclusion constraint may be reformulated into a chance constraint with the final result inheriting the associated probabilistic guarantees.
\end{remark}
\begin{remark}
Indeed, any $k$-sparse feasible vector $\mathbf{e}\in\mathbb{R}^m$, $\abs{\supp(\mathbf{e})}\le k<m$  will be recovered exactly by the solution to the optimization problem in \eqref{eqn:comp_sens_subspace}. Although the question of the stability of the recovery process to process noise is not pursued in this paper, we expect similar saturated error bound results as obtained above. We will demonstrate the stability numerically by including noise in the example given in subsequent sections.
\end{remark}

\begin{proof}
Let $\mathbf{e}$ be a feasible point of the optimization problem in \eqref{eqn:comp_sens_subspace}, and $\hat{\mathbf{e}}\triangleq\mathbf{e} + \mathbf{h}, \hspace{1mm} \mathbf{h}\in\mathbb{R}^n$  be the optimal point. Given $k<m$, define the index set $\mathcal{T}\subset\left\{1,2,\hdots,m\right\}$ with $\abs{\mathcal{T}}\le k$. By the optimality of $\hat{\mathbf{e}}$, we have that $\norm{\mathbf{e}}_1\ge\norm{\hat{\mathbf{e}}}_1$, which implies that:
\begin{align*}
    \norm{\mathbf{e}}_1&\ge\norm{\hat{\mathbf{e}}}_1 = \norm{\mathbf{e}+\mathbf{h}}_1 \\
    &\phantom{norm{\hat{\mathbf{e}}}}=\norm{\mathbf{e}_\mathcal{T}+\mathbf{h}_\mathcal{T}}_1 +  \norm{\mathbf{e}_{\mathcal{T}^c}+\mathbf{h}_{\mathcal{T}^c}}_1\\
    &\ge\norm{\mathbf{e}_{\mathcal{T}}}_1 - \norm{\mathbf{h}_{\mathcal{T}}}_1 + \norm{\mathbf{h}_{\mathcal{T}^c}}_1 - \norm{\mathbf{e}_{\mathcal{T}^c}}_1\\
    \Longrightarrow \norm{\mathbf{h}_{\mathcal{T}^c}}_1&\le \norm{\mathbf{h}_{\mathcal{T}}}_1 + \norm{\mathbf{e}}_1 - \norm{\mathbf{e}_{\mathcal{T}}}_1 + \norm{\mathbf{e}_{\mathcal{T}^c}}_1\\
    &= \norm{\mathbf{h}_\mathcal{T}}_1 + 2\norm{\mathbf{e}_{\mathcal{T}^c}}_1.
\end{align*}
Thus
\begin{align}\label{eqn:claim1}
\norm{\mathbf{h}_{\mathcal{T}^c}}_1\le\norm{\mathbf{h}_\mathcal{T}}_1 + 2\norm{\mathbf{e}_{\mathcal{T}^c}}_1.
\end{align}

\noindent Next, since $\mathbf{e}$ and $\hat{\mathbf{e}}$ are feasible, i.e., $\mathbf{e},\hat{\mathbf{e}}\in\mathcal{X}\Rightarrow\norm{\mathbf{e}-\mathbf{y}}_1\le\delta$ and $\norm{\hat{\mathbf{e}}-\mathbf{y}}_1 = \norm{\mathbf{h}+\mathbf{e}-\mathbf{y}}_1\le\delta$, it follows that
\begin{align}\label{eqn:claim2}
    \norm{\mathbf{h}}_1 = \norm{\mathbf{h}_\mathcal{T}}_1 + \norm{\mathbf{h}_{\mathcal{T}^c}}_1&\le2\delta.
\end{align}

\noindent Moreover, from the feasibility of $\mathbf{e}$ and $\hat{\mathbf{e}}$, $\mathbf{e}-\mathbf{y},\hat{\mathbf{e}}-\mathbf{y}\in\mathcal{V}\Rightarrow\mathbf{h}=\hat{\mathbf{e}}-\mathbf{e}\in\mathcal{V}$. Thus, from the subspace property, it follows that
\begin{align}\label{eqn:claim3}
    \norm{\mathbf{h}_\mathcal{T}}_1\le\gamma\norm{\mathbf{h}_{\mathcal{T}^c}}_1,\hspace{2mm}\text{for some } 0<\gamma<1.
\end{align}

Adding the inequalities in \eqref{eqn:claim1} and \eqref{eqn:claim3} gives
\begin{align*}
    \norm{\mathbf{h}}_1&\le\left(1-\gamma\right)\norm{\mathbf{h}_\mathcal{T}}_1 + \gamma\norm{\mathbf{h}_\mathcal{T}}_1 + 2\norm{\mathbf{e}_{\mathcal{T}^c}}_1 + \gamma\norm{\mathbf{h}_{\mathcal{T}^c}}_1\\
    &\le\left(1-\gamma\right)\norm{\mathbf{h}_\mathcal{T}}_1 + \gamma\norm{\mathbf{h}}_1 + 2\norm{\mathbf{e}_{\mathcal{T}^c}}_1.
\end{align*}
Subtracting $\gamma\norm{\mathbf{h}}_1$ from both sides and dividing by $1 - \gamma$ gives
\begin{align*}
    \norm{\mathbf{h}}_1\le\norm{\mathbf{h}_\mathcal{T}}_1 + \frac{2}{1-\gamma}\norm{\mathbf{e}_{\mathcal{T}^c}}_1,
\end{align*}
so that
\begin{align}\label{eqn:claim4}
    \norm{\mathbf{h}_{\mathcal{T}^c}}_1\le \frac{2}{1-\gamma}\norm{\mathbf{e}_{\mathcal{T}^c}}_1.
\end{align}
Combining \eqref{eqn:claim3} and \eqref{eqn:claim4} yields
\begin{align}\label{eqn:claim5}
    \norm{\mathbf{h}_\mathcal{T}}_1\le\gamma\norm{\mathbf{h}_{\mathcal{T}^c}}_1 \le \frac{2\gamma}{1-\gamma}\norm{\mathbf{e}_{\mathcal{T}^c}}_1.
\end{align}
By, adding the inequalities in \eqref{eqn:claim4} and \eqref{eqn:claim5}, it follows that
\begin{align}\label{eqn:claim6}
    \norm{\mathbf{h}}_1=\norm{\mathbf{h}_\mathcal{T}}_1+\norm{\mathbf{h}_{\mathcal{T}^c}}\le \frac{2\left(1+\gamma\right)}{1-\gamma}\norm{\mathbf{e}_{\mathcal{T}^c}}_1,
\end{align}
which, after combining with \eqref{eqn:claim2}, yields
\begin{align*}
    \norm{\mathbf{h}}_1&\le\min\left\{\frac{2\left(1+\gamma\right)}{1-\gamma}\norm{\mathbf{e}_{\mathcal{T}^c}}_1,2\delta\right\}\\
    &\le2\min\left\{\frac{\left(1+\gamma\right)}{1-\gamma}\norm{\mathbf{e}_{\mathcal{T}^c}}_1,\delta\right\}.
\end{align*}
Thus, the inequality
\begin{align*}
    \norm{\mathbf{h}}_1\le 2\textsf{sat}_\delta\left(\frac{\left(1+\gamma\right)}{1-\gamma}\norm{\mathbf{e}_{\mathcal{T}^c}}_1\right)
\end{align*} holds for all index set $\mathcal{T}\subset\left\{1,2,\hdots,m\right\}$ with $\abs{\mathcal{T}}\le k$. The result follows by selecting $\mathcal{T} = \supp(\mathbf{e})$.
\end{proof}

Now, we focus on the development of a resilient reconstruction algorithm using both  measurement model and a prior information model. Consider a concurrent model of the form:
\begin{align}\label{eqn:model_based}
    \mathbf{y} &= H\mathbf{x} + \mathbf{e}+ \boldsymbol{\varepsilon}\\\label{eqn:data_driven}
    \mathbf{y} &\sim\mathcal{N}(\mu(\mathbf{z}),\Sigma(\mathbf{z}))\\\label{eqn:noise_model}
    \boldsymbol{\varepsilon}&\sim\mathcal{N}(\mathbf{0},\textsf{diag}(\sigma_1^2,\hdots,\sigma_m^2))
\end{align}
where $H\in\mathbb{R}^{m\times n}$ is the measurement matrix, $\mathbf{x}\in\mathbb{R}^n$ is the state vector, $\mathbf{e}\in\mathbb{R}^m, \norm{\mathbf{e}}_0\le k<m$ is the attack vector, and $\boldsymbol{\varepsilon}\in\mathbb{R}^m$ is the measurement noise. The concurrent model consists of a measurement model \eqref{eqn:model_based}, prior information (auxiliary) model \eqref{eqn:data_driven} given as a function of the auxiliary variable $\mathbf{z}\in\mathbb{R}^p$, and a noise model \eqref{eqn:noise_model}, 
\noindent where
\begin{align*}
    \mu(\mathbf{z}) = \left[\begin{array}{c}\mu_1(\mathbf{z})\\\vdots\\\mu_m(\mathbf{z})\end{array}\right]\text{ and }
    \Sigma(\mathbf{z})= \left[\begin{array}{ccc}\Sigma_1(\mathbf{z})&&\\&\ddots&\\&&\Sigma_m(\mathbf{z})\end{array}\right]
\end{align*}
for some mean and covariance functions $\mu_i:\mathbb{R}^p\mapsto\mathbf{R}$ and $\Sigma_i:\mathbb{R}^p\mapsto\mathbb{R}_+$ respectively (see Section~\ref{sec:auxiliary_model} for a particular example using GPR). For a Cyber-physical system, the measurement model is usually physics-based while the prior-information is data-driven. The noise model is generally knowledge-based. One of the main advantages of using models of this form for a CPS is that the resulting blend of the generalization properties of physics-based models and the adaptive local accuracy of data-driven methods creates an additional layer of redundancy which can reveal the truth even if portions of the measurement is subject to adversarial corruption. In order to remain undetectable, any viable attack vector $\mathbf{y}_a,\norm{\mathbf{y}_a}_{\ell_0}=p\le m$ necessarily have to satisfy the condition $p(\mathbf{y}+\mathbf{y}_a|\mathbf{z},\mathcal{D})\ge p(\mathbf{y}|\mathbf{z},\mathcal{D})$. This provides an additional layer of security by: \textit{1)} requiring the attacker to have knowledge of the auxiliary model and the parameters, and \textit{2)} limiting the magnitude of possible state corruption.

Let $\mathbf{y}^*$ be the true value of the measured variable, the optimal estimation problem is cast as the optimization problem:

\begin{align}\label{eqn:enh_res_est}
\begin{array}{lc}
\Minimize & \left\|\mathbf{y}-H\mathbf{x}-\boldsymbol{\varepsilon}\right\|_{l_0} \\
\SubjectTo&\\
          &H\mathbf{x}\in\mathcal{Y}(\mathbf{z})\\
          &\boldsymbol{\varepsilon}\in\mathcal{E},
\end{array}
\end{align}
where the convex sets $\mathcal{Y}(\mathbf{z})$ and $\mathcal{E}$ have the property that:
\begin{align}
    p(\mathbf{y}^*\in\mathcal{Y}|\mathbf{z},\mathcal{D})\ge\tau\\
          p(\boldsymbol{\varepsilon}^*\in\mathcal{E})\ge\tau.
\end{align}

\noindent The idea is essentially seeking a state vector, together with the minimum attacked channels and a highly likely noise vector, which completely explains the observations while having a high likelihood according to the auxiliary model prior. Ideally, one would use an index minimizing ``$0$-norm" in the objective, as done above. However, Theorem~\ref{thm:main_result} shows that the $1$-norm relaxation achieves a really good reconstruction property, provided that the range space of $H$ satisfies the subspace property. The optimization parameter $\tau\in(0,\hspace{2mm}1]$ controls the likelihood threshold. It can be set to a constant value or optimized with respect to some higher-level objectives. Thus, the resilient state estimation optimization problem is equivalent to:

\begin{align}\label{eqn:enh_res_est2}
\begin{array}{ll}
\Minimize & \left\|\mathbf{y}-H\mathbf{x}-\boldsymbol{\varepsilon}\right\|_1 \\
\SubjectTo&\\
        &\begin{array}{rl}
          \norm{H\mathbf{x}+\boldsymbol{\varepsilon}-\mu(\mathbf{z})}_{\Sigma^{-1}(\mathbf{z})}^2 & \le \chi^2_m(\tau)\\
          \norm{\boldsymbol{\varepsilon}}_{\Sigma_\varepsilon^{-1}}^2 & \le\chi^2_{m}(\tau),
        \end{array}
\end{array}
\end{align}

\noindent where $\Sigma_\varepsilon = \textsf{diag}(\sigma_1^2,\hdots,\sigma_m^2)$ and $\chi^2_m(\tau)$ is the quantile function for probability $\tau$ of the chi-squared distribution with $m$ degrees of freedom.

The following lemma will be useful in proving the next result about the reconstruction error bound of the resulting resilient estimation based on the optimization problem in \eqref{eqn:enh_res_est2}.
\begin{lemma}\label{lemma:kth_term_error_bound}
Given a vector $\boldsymbol{\varepsilon}\in\mathbb{R}^m$ with $\norm{\boldsymbol{\varepsilon}}_2\le\delta$, then the following $k$th term approximation error bound
\begin{align}
    \norm{\boldsymbol{\varepsilon} - \boldsymbol{\varepsilon}[k]}_1\le\frac{m-k}{\sqrt{m}}\delta
\end{align}
holds for $k<m$.
\end{lemma}
\begin{proof}
Without loss of generality, suppose the elements of $\boldsymbol{\varepsilon}$ are ordered as $\abs{\boldsymbol{\varepsilon}_1}\le\abs{\boldsymbol{\varepsilon}_2}\le\hdots\le\abs{\boldsymbol{\varepsilon}_m}$, then
\begin{align*}
    \norm{\boldsymbol{\varepsilon} - \boldsymbol{\varepsilon}[k]}_1 &= \sum\limits_{i=1}^{m-k}{\abs{\boldsymbol{\varepsilon}_i}}\\
                                                                    &\le\sum\limits_{i=1}^{m}{\abs{\boldsymbol{\varepsilon}_i}} - k\abs{\boldsymbol{\varepsilon}_k} = \norm{\boldsymbol{\varepsilon}}_1-k\abs{\boldsymbol{\varepsilon}_k}\\&\le \norm{\boldsymbol{\varepsilon}}_1-\frac{k}{m-k}(m-k)\abs{\boldsymbol{\varepsilon}_k}\\
                                                                    &\le\norm{\boldsymbol{\varepsilon}}_1-\frac{k}{m-k}\sum\limits_{i=1}^{m-k}{\abs{\boldsymbol{\varepsilon}_i}}\\&\le \norm{\boldsymbol{\varepsilon}}_1-\frac{k}{m-k}\norm{\boldsymbol{\varepsilon} - \boldsymbol{\varepsilon}[k]}_1
\end{align*}
From which
\begin{align*}
    \norm{\boldsymbol{\varepsilon} - \boldsymbol{\varepsilon}[k]}_1\le\frac{m-k}{m}\norm{\boldsymbol{\varepsilon}}_1\le\frac{m-k}{\sqrt{m}}\norm{\boldsymbol{\varepsilon}}_2\le\frac{m-k}{\sqrt{m}}\delta
\end{align*}
\end{proof}

\begin{theorem}
Consider the recovery optimization problem in \eqref{eqn:enh_res_est2}. Suppose the unknown true state $\mathbf{x}^*\in\mathbb{R}^n$ is a feasible of the optimization problem. If the range space $\mathcal{R}(H)$ of $H$ satisfies the \emph{subspace property} $\norm{\mathbf{v}_\mathcal{T}}_1\le\gamma\norm{\mathbf{v}_{\mathcal{T}^c}}_1$, $\forall \mathbf{v}\in\mathcal{R}(H), \abs{\mathcal{T}}\le k<m$, then the reconstruction error can be upper bounded as:
\begin{align}\nonumber
     \norm{\hat{\mathbf{x}} - \mathbf{x}^*}_2&\le C_1\textsf{sat}_{\delta(\tau)}\left(C_2\norm{\hat{\mathbf{e}}-\hat{\mathbf{e}}[k] }_1 + C_3\delta(\tau)\right) \\&\hspace{3cm}+ C_1\textsf{sat}_{\delta(\tau)}\left(C_3\delta(\tau)\right),
\end{align}
where $\hat{\mathbf{e}} = \mathbf{y}-H\hat{\mathbf{x}} - \hat{\boldsymbol{\varepsilon}}$ is the objective residual,
\begin{align*}
   \delta(\tau) &= \overline{\Sigma}^{\frac{1}{2}}\chi_m(\tau), \hspace{1mm}
    C_1 = \frac{2}{\underline{{\sigma}}_H},\hspace{1mm}
    C_2 =\frac{1+\gamma}{1-\gamma},\hspace{1mm}\\
    &C_3 = \frac{(1+\gamma)}{(1-\gamma)}\frac{(m-k)}{\sqrt{m}}\overline{\sigma},
\end{align*}
$\underline{{\sigma}}_H$ is the smallest singular value of $H$, and $\overline{\sigma}$ and $\overline{\Sigma}$ are the biggest standard deviations of the auxiliary model and measurement noise statistics respectively.
\end{theorem}

\begin{proof}
Define the sets $\mathcal{X}, \mathcal{X}_\varepsilon\subset\mathbb{R}^m$ as
\begin{align*}
    \mathcal{X}(\mathbf{z}) &\triangleq \left\{\mathbf{y}\in\mathbb{R}^m\hspace{1mm}:\hspace{1mm}\norm{\mathbf{y}-\boldsymbol{\mu}(\mathbf{z})}_{\Sigma^{-1}(\mathbf{z})}^2\le\chi_m^2(\tau)\right\}\\
    \mathcal{X}_\varepsilon&\triangleq\left\{\boldsymbol{\varepsilon}\in\mathbb{R}^m\hspace{1mm}:\hspace{1mm}\norm{\boldsymbol{\varepsilon}}_{\Sigma_{\boldsymbol{\varepsilon}}^{-1}}^2\le\chi_m^2(\tau)\right\}.
\end{align*}
Thus, the optimization problem in \eqref{eqn:enh_res_est2} can be expressed as:
\begin{align}\label{eqn:enh_res_est_P}\tag{$P$}
\begin{array}{ll}
\Minimize & \left\|\mathbf{e}\right\|_1 \\
\SubjectTo&\\
        &\begin{array}{rl}
          \mathbf{y}-\mathbf{e}-\boldsymbol{\varepsilon}&\in\mathcal{R}(H)\\
          \mathbf{y}-\mathbf{e}&\in\mathcal{X}(\mathbf{z})\\
          \boldsymbol{\varepsilon}&\in\mathcal{X}_{\boldsymbol{\varepsilon}}.
        \end{array}
\end{array}
\end{align}
Also, consider the reduced problem
\begin{align}\label{eqn:enh_res_est_P_hat}\tag{$\hat{P}$}
\begin{array}{ll}
\Minimize & \left\|\mathbf{e}\right\|_1 \\
\SubjectTo&\\
        &\mathbf{y}-\mathbf{e}\in\mathcal{R}(H)\cap\mathcal{X}.
\end{array}
\end{align}
Let 
\begin{itemize}
    \item $\mathbf{e}^*\in\mathbb{R}^m,\hspace{2mm} \norm{\mathbf{e}^*}_0=k$ and $\boldsymbol{\varepsilon}^*\in\mathbb{R}^m$ be the unknown actual attack vector and noise instance respectively,
    \item $\hat{\mathbf{e}}, \hat{\boldsymbol{\varepsilon}}\in\mathbb{R}^m$ be the minimal points of the optimization problem in \eqref{eqn:enh_res_est_P}, and
    \item $\hat{\mathbf{e}}_2\in\mathbb{R}^m$ be the solution of the reduced problem in \eqref{eqn:enh_res_est_P_hat}.
\end{itemize}
Using the result in Theorem~\ref{thm:main_result}, the observation that $\mathbf{e}^*+\boldsymbol{\varepsilon}^*$ and $\hat{\mathbf{e}} + \hat{\boldsymbol{\varepsilon}}$ are feasible points of \eqref{eqn:enh_res_est_P_hat} and Lemma~\ref{lemma:kth_term_error_bound}, yield:
\begin{align*}
    \norm{\hat{\mathbf{e}}_2 - \mathbf{e}^* - \boldsymbol{\varepsilon}^*}_1&\le 2\textsf{sat}_\delta\left(\frac{1+\gamma}{1-\gamma}\norm{\boldsymbol{\varepsilon}^* - \boldsymbol{\varepsilon}[k]^*}_1\right)\\&\le2\textsf{sat}_\delta\left(\frac{(1+\gamma)(m-k)\bar{\Sigma}}{(1-\gamma)\sqrt{m}}\delta\right),
\end{align*}
with $\delta = \overline{\Sigma}^{\frac{1}{2}}\chi_m(\tau)$. Using the left-hand-side triangular inequality, the above inequality implies that:
\begin{align*}
    &\norm{\hat{\mathbf{e}} + \hat{\boldsymbol{\varepsilon}} - \mathbf{e}^* - \boldsymbol{\varepsilon}^*}_1\le \norm{\hat{\mathbf{e}}_2 - \mathbf{e}^* - \boldsymbol{\varepsilon}^*}_1 + \norm{\hat{\mathbf{e}}_2 - \hat{\mathbf{e}} - \hat{\boldsymbol{\varepsilon}}}_1\\
    &\le2\textsf{sat}_\delta\left(\frac{(1+\gamma)(m-k)\bar{\Sigma}}{(1-\gamma)\sqrt{m}}\delta\right)\\&\hspace{1cm} + 2\textsf{sat}_\delta\left(\frac{1+\gamma}{1-\gamma}\norm{\hat{\mathbf{e}}-\hat{\mathbf{e}}[k] + \hat{\boldsymbol{\varepsilon}} - \hat{\boldsymbol{\varepsilon}}[k]}_1\right)\\
    &\le2\textsf{sat}_\delta\left(\frac{(1+\gamma)(m-k)\bar{\Sigma}}{(1-\gamma)\sqrt{m}}\delta\right) \\&\hspace{5mm}+ 2\textsf{sat}_\delta\left(\frac{1+\gamma}{1-\gamma}\norm{\hat{\mathbf{e}}-\hat{\mathbf{e}}[k] }_1 + \frac{(1+\gamma)(m-k)\bar{\Sigma}}{(1-\gamma)\sqrt{m}}\delta\right).
\end{align*}
Expressing the right-hand-side of the last inequality in the ``language" of the original problem in \eqref{eqn:enh_res_est2} yields
\begin{align*}
    &\norm{H(\hat{\mathbf{x}} - \mathbf{x}^*)}_1\le2\textsf{sat}_\delta\left(\frac{(1+\gamma)(m-k)\bar{\Sigma}}{(1-\gamma)\sqrt{m}}\delta\right) \\&\hspace{1cm}+ 2\textsf{sat}_\delta\left(\frac{1+\gamma}{1-\gamma}\norm{\hat{\mathbf{e}}-\hat{\mathbf{e}}[k] }_1 + \frac{(1+\gamma)(m-k)\bar{\Sigma}}{(1-\gamma)\sqrt{m}}\delta\right),
\end{align*}
where $\hat{\mathbf{x}} - \mathbf{x}^*$ is the resulting state estimation error, which is consequently bounded as:
\begin{align*}
    \norm{\hat{\mathbf{x}} - \mathbf{x}^*}_2\le C_1\textsf{sat}_\delta\left(C_2\norm{\hat{\mathbf{e}}-\hat{\mathbf{e}}[k] }_1 + C_3\delta\right) + C_1\textsf{sat}_\delta\left(C_3\delta\right)
\end{align*}
\end{proof}

\section{Numerical Example: Power system state estimation with data-driven economic auxiliary model}\label{s:experiments}
In this numerical simulation example, a resilient state estimation algorithm based on the optimization problem in \eqref{eqn:enh_res_est2} is developed and evaluated on the IEEE 14-bus test case mapped to actual data from the the New York Independent System Operator (NYISO). For this application, the prior information is obtained from a GPR mapping from some energy market information to an \emph{iid} Gaussian distributions on the system measurements. This example first appeared in our earlier work\cite{anubi2019enhanced}. Interested readers are directed to that paper for more details. In what follows, we only provide an overview to strengthen the theoretical results of the previous sections. 

\subsection{Setup}\label{ss:setup}
The IEEE 14-bus system, shown in Fig. \ref{fig:IEEE14bus},  represents a simple approximation of the American electric power system as of February 1962. It has 14 buses, 5 generators, and 11 loads. The system has 27 state variables which are the voltage angles and voltage magnitudes of the buses, with the first bus angle chosen as the reference one. The buses/nodes of the power grid model are assumed to be supported with IIoT measurement sensors such as remote terminal units (RTUs) able to provide bus-related measurements of active and reactive power injection and flow. 

Simulation experiments are performed using the actual load data of New York state as provided by NYISO\cite{NYISO_data}. Specifically, five-minute load data of NYISO for 3 months (between January and March) in 2017 and 2018 are used. Furthermore, each region of the NYISO map, shown in Fig. \ref{fig:NYISO}, is mapped in an ascending order with every load bus of IEEE 14 system, i.e. using the following mapping: $[2\rightarrow1,~3\rightarrow2,~4\rightarrow3,~5\rightarrow4,~6\rightarrow5,~9\rightarrow6,~10\rightarrow7,~11\rightarrow8,~12\rightarrow9,~13\rightarrow10,~14\rightarrow11] $, where the first element show the load bus of IEEE 14 case the second the region of NYISO, e.g., bus 2 to region A-WEST, bus 3 to region B-GENESE, bus 4 to region C-CENTRL, etc. By this, we were able to create realistic attack data to validate the earlier theoretical claims.

\begin{figure}[t]
    \centering
    \begin{subfigure}{0.5\textwidth}
        \centering{\includegraphics[width=3.1in]{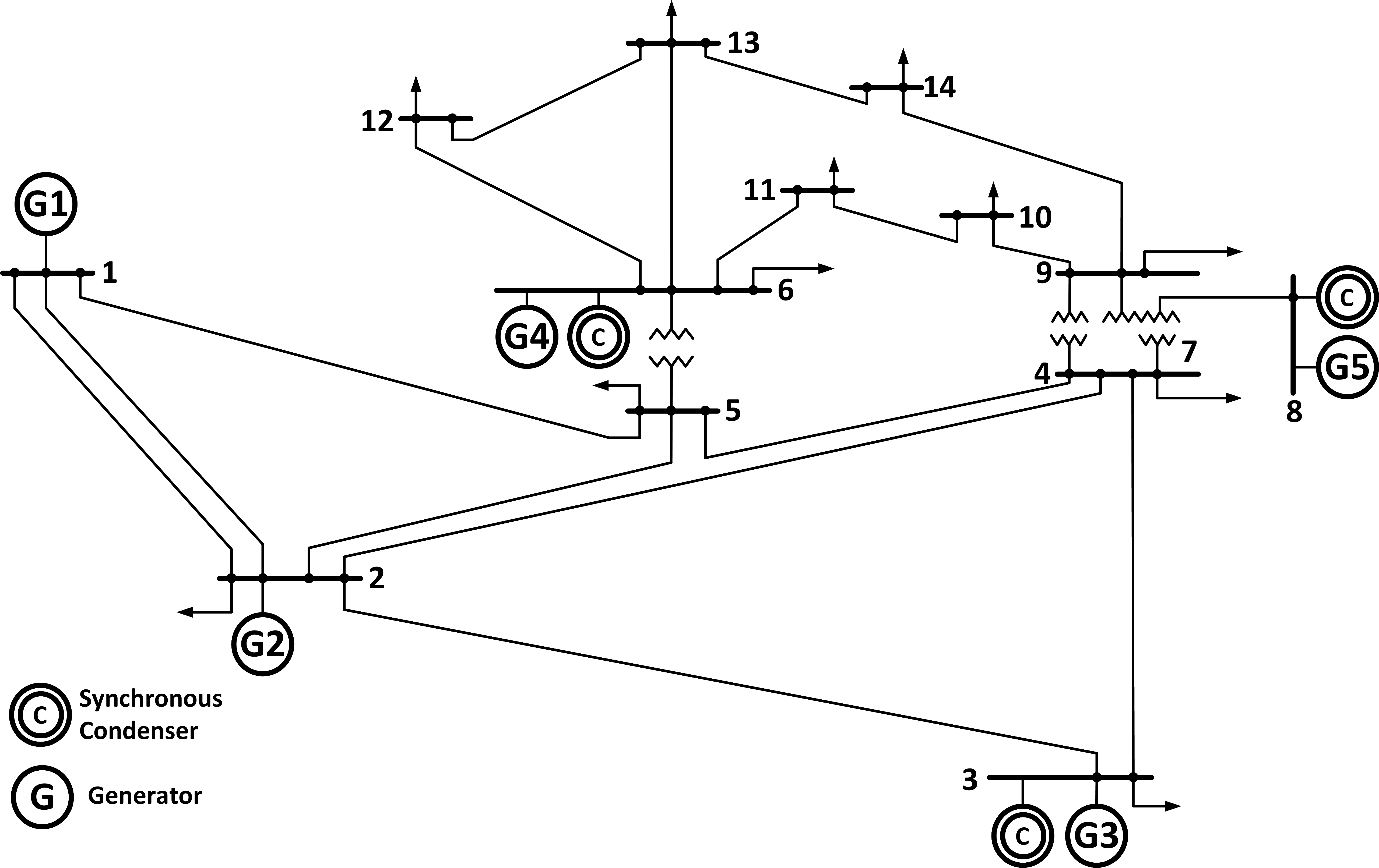}}
        \caption{IEEE 14-bus system.}
        \label{fig:IEEE14bus}
    \end{subfigure}%
    \par\bigskip\bigskip
    \begin{subfigure}{0.5\textwidth}
        \centering{\includegraphics[width=3.1in]{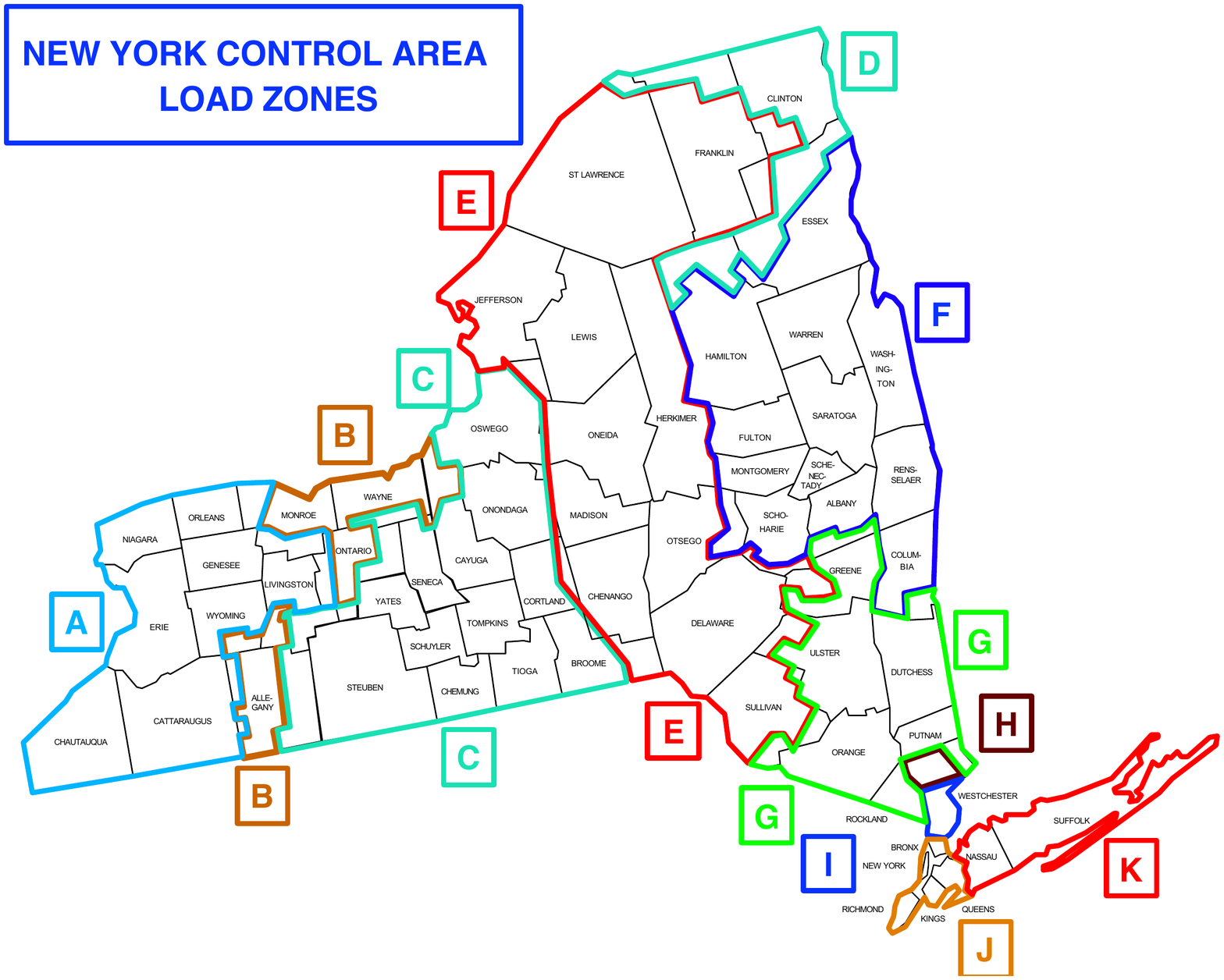}}
        \caption{NYISO map of the 11 control area load zones}
        \label{fig:NYISO}
    \end{subfigure}
    \caption{IEEE 14-bus system mapped into NYISO control area load zones data.}
\end{figure}

\subsection{Auxiliary model}\label{sec:auxiliary_model}
From the collected NYSIO historical load and market data, we built a Gaussian Process Regression (GPR) model which maps from locational bus marginal prices to bus voltages and angle measurements. This, as shown in previous sections, provides an added layer of redundancy for boosting system resiliency to arbitrary data corruption. A Gaussian Process (GP) is a collection (possibly infinite) of continuous random variables $\mathcal{G}$, any finite subset of which are jointly Gaussian.  GPR uses GPs to encode prior distributions over functions\footnote{In this case will be functions from auxiliary measurements to observed measurements.}. The priors are then updated to form posterior distributions when new data is collected. For a comprehensive introduction to GP and GPR, and their applications for learning and control, the readers are directed to \cite{rasmussen2006gaussian} and a recent survey in \cite{Liu2018Gaussian}.

Consider a dataset $\mathcal{D} = \left\{\mathbf{Z},\mathbf{Y}\right\}$, where $\mathbf{Z}\in\mathbb{R}^{p\times N}$ is a matrix containing the values of the auxiliary variables column-wise, $\mathbf{Y}\in\mathbb{R}^{m\times N}$ are the corresponding sensor measurement values and $N$ is the number of datapoint in the dataset. The goal is to learn an implicit mapping $f:\mathbb{R}^p\mapsto\mathbb{R}^m$ for which

\begin{align}
    \mathbf{y}_i = f(\mathbf{z}_i) + \boldsymbol{\varepsilon},\hspace{2mm} i=1,\hdots N,
\end{align}

\noindent where $\boldsymbol{\varepsilon}\sim\mathcal{N}(\mathbf{0},\textsf{diag}(\sigma_1^2,\hdots,\sigma_m^2))$. In theory, without any further restriction, the problem is ill-defined because there are potentially many possible functions that explains the data exactly notwithstanding the measurement noise. As a means of regularization, the class of functions for consideration is refined by the restriction $f(\mathbf{z})\sim\mathcal{GP}(m(\mathbf{z}),k(\mathbf{z},\mathbf{z}'))$ to a GP completely specified by its mean and covariance functions\footnote{Also known as kernels.}

\begin{align}
    \mu(\mathbf{z}) &\triangleq \mathbb{E}[f(\mathbf{z})]\\
    k(\mathbf{z},\mathbf{z}') &\triangleq \mathbb{E}[(f(\mathbf{z})-\mu(
    \mathbf{z}))(f(\mathbf{z'})-\mu(
    \mathbf{z}'))].
\end{align}

\noindent The covariance function can then be specified apriori without an explicit probability distribution. This is where the prior (possibly knowledge-based) information is encoded in the GP. While any positive definite function may pass for a covariance function, one commonly used  is the squared exponential covariance function:

\begin{align}
    k(\mathbf{z},\mathbf{z}') = A\exp\left(-\frac{1}{2l}\norm{\mathbf{z}-\mathbf{z}'}_2^2\right),
\end{align}

\noindent where hyperparameters $A$ and $l$ implicitly define a smoothness-promoting prior. Given a query point $\mathbf{z}\in\mathbb{R}^p$ for the auxiliary variable, the posterior distribution for the $j$th sensor values is $p(y_j|\mathbf{z},\mathcal{D}) = \mathcal{N}(\mu_j(\mathbf{z}),\Sigma_j(\mathbf{z}))$, with the mean and covariance function given by

\begin{align}
    \mu_j(\mathbf{z}) &= \mathbf{k}(\mathbf{z})^\top\left(K+\sigma_j^2 I\right)^{-1}\mathbf{Y}_j^\top,\\
    \Sigma_j(\mathbf{z})&= k(\mathbf{z},\mathbf{z}) - \mathbf{k}(\mathbf{z})^\top\left(K+\sigma_j^2 I\right)^{-1}\mathbf{k}(\mathbf{z}),\hspace{2mm}j=1,\hdots,m
\end{align}

\noindent where $K\in\mathbb{R}^{N\times N}$ is a covariance matrix with entries $K_{ij}=k(\mathbf{z}_i,\mathbf{z}_j)$ and $\mathbf{k}(\mathbf{z})\in\mathbb{R}^N$ is a vector with entries $\mathbf{k}(\mathbf{z})_i = k(\mathbf{z},\mathbf{z}_i)$.

The overall sensor values posterior distribution is given by:

\begin{align}
    p(\mathbf{y}|\mathbf{z},\mathcal{D}) &= \prod_{j=1}^m \mathcal{N}(\mu_j(\mathbf{z}),\Sigma_j(\mathbf{z}))\\
                                         &=\mathcal{N}(\mu(\mathbf{z}),\Sigma(\mathbf{z})),
\end{align}

\noindent where

\begin{align*}
    \mu(\mathbf{z}) = \left[\begin{array}{c}\mu_1(\mathbf{z})\\\vdots\\\mu_m(\mathbf{z})\end{array}\right]\text{ and }
    \Sigma(\mathbf{z})= \left[\begin{array}{ccc}\Sigma_1(\mathbf{z})&&\\&\ddots&\\&&\Sigma_m(\mathbf{z})\end{array}\right]
\end{align*}

\subsection{Solution Algorithm}
In addition to the nice reconstruction property of the $1$-norm relaxation, Iteratively re-weighted algorithms \cite{candes2008enhancing, chartrand2008iterative} have been demonstrated to be a highly effective way of approximating the solution of the nonconvex problem with successive convex problems. In particular, for the solution of the problem in \eqref{eqn:enh_res_est}, the re-weighted $1$-norm minimization scheme of \cite{candes2008enhancing} is employed to give even stronger reconstruction algorithm.

Consider the operator $\mathcal{P}:\mathbb{R}^m\times\mathbb{R}^p\times\mathbb{R}^{m\times m}\mapsto \mathbb{R}^{n+m}$, where

\begin{align}
    \hat{\mathbf{x}}(W),\hspace{2mm}\hat{\boldsymbol{\varepsilon}}(W) =  \mathcal{P}(\mathbf{y},\mathbf{z},W)
\end{align}

\noindent are given by the minimizers of the convex program:

\begin{align}\label{eqn:enh_res_est3}
\begin{array}{lc}
\Minimize & \left\|W\left(\mathbf{y}-H\mathbf{x}-\boldsymbol{\varepsilon}\right)\right\|_1 \\
\SubjectTo&\\
        &\begin{array}{rl}
          \norm{H\mathbf{x}+\boldsymbol{\varepsilon}-\mu(\mathbf{z})}_{\Sigma^{-1}(\mathbf{z})}^2 & \le \chi^2_{n_c}(\tau)\\
          \norm{\boldsymbol{\varepsilon}}_{\Sigma_\varepsilon^{-1}}^2 & \le\chi^2_{m}(\tau),
        \end{array}
\end{array}
\end{align}

\noindent Using this, the algorithm for the enhance state estimator is outlined in Algorithm~\ref{alg:enh_state_est}.

\begin{algorithm}[htbp]

\caption{Resilient Optimal State Estimation Algorithm Using Re-weighted $1$-norm minimization}\label{alg:enh_state_est}
\begin{algorithmic}
\Procedure{Offline}{}
	\State $\mathcal{D} \gets \left\{\mathbf{Z},\mathbf{Y}\right\}$\Comment{Dataset sparsification}
	\State $K\gets k(\mathbf{Z},\mathbf{Z})$\Comment{Kernel matrix}
	\State $\Sigma_\varepsilon,A,l \gets$\Comment{Hyperparameters initialization,}
\EndProcedure

\Procedure{Collect Data}{}
	\State $\mathbf{y}\gets$\Comment{Sensor measurements at the current instant}
	\State $\mathbf{z}\gets$\Comment{Auxiliary measurements at the current instant}
\EndProcedure

\Procedure{Update Models}{}
    \State $H\gets$\Comment{Model-based. See Sub-section~\ref{ss:setup} for details}
    \For{$j = 1$ to $m$}\Comment{Data-driven posterior}
        \State $\mu_j \gets \mathbf{k}(\mathbf{z})^\top\left(K+\sigma_j^2 I\right)^{-1}\mathbf{Y}_j^\top,$\Comment{Mean}
        \State $\Sigma_j\gets k(\mathbf{z},\mathbf{z}) - \mathbf{k}(\mathbf{z})^\top\left(K+\sigma_j^2 I\right)^{-1}\mathbf{k}(\mathbf{z}),\hspace{2mm}$\Comment{Covariance}
    \EndFor
\EndProcedure

\Procedure{Re-weighted $1$-norm minimization}{$\mathbf{y}$,$\mathbf{z}$}
    \State $W\triangleq \textsf{diag}[w_1,\hdots,w_m]\gets I$
    \State $l\gets0$\Comment{Iteration count}
    \While{\texttt{not converged and }$l\le l_{max}$}
    \State $\hat{\mathbf{x}}^l,\hspace{2mm}\hat{\boldsymbol{\varepsilon}}^l \gets  \mathcal{P}(\mathbf{y},\mathbf{z},W)$\Comment{$\ell_1$ minimization}
    \State $\mathbf{r}\gets \mathbf{y}-H\hat{\mathbf{x}}^l - \hat{\boldsymbol{\varepsilon}}^l$\Comment{residual}
    \For{$j=1$ to $m$}\Comment{weights update}
        \State $w_j\gets\frac{1}{\abs{\mathbf{r}_j}+\delta}$
    \EndFor
    $l\gets l+1$\Comment{increment counter}
    \EndWhile\label{euclidendwhile}
    \State \textbf{return} $\hat{\mathbf{x}}^l,\hat{\boldsymbol{\varepsilon}}^l$\Comment{State estimate is $\hat{\mathbf{x}}^l$}
   
\EndProcedure
\end{algorithmic}

\end{algorithm}

\subsection{Results}
The enhanced resilient estimation algorithm in Algorithm~\ref{alg:enh_state_est} was implemented and ran for data collected every five minutes in a simulation environment. The process begins with the auxiliary measurements $\mathbf{z}=\left[\begin{array}{ccc}z_\text{lbmp}&z_\text{mcl}&z_\text{mcc}\end{array}\right]$, which are actual data downloaded from the respective nodes of the NYISO transmission grid. Here, $z_\text{lbmp}$ is the \texttt{locational bus marginal prices} (\$/MWh), $z_\text{mcl}$ is the \texttt{marginal cost loses} (\$/MWh) and $z_\text{mcc}$ is the \texttt{marginal cost congestion} (\$/MWh). Next, the trained GPR model is executed to give the mean $\boldsymbol{\mu}(\mathbf{z})$ and the covariance $\Sigma(\mathbf{z})$ of the data-driven auxiliary model. Two kinds of FDIA generation were used in the simulation. For the first kind, attack vectors are generated to bias selected measurements locations by 500\% of its true value along a randomly chosen direction. For the second kind, the attack vectors $\mathbf{y}_a$ are systematically generated to result in a specified bias in the state estimation at targeted state variables. 

\begin{figure}[t]
    \centering{\includegraphics[width=3in]{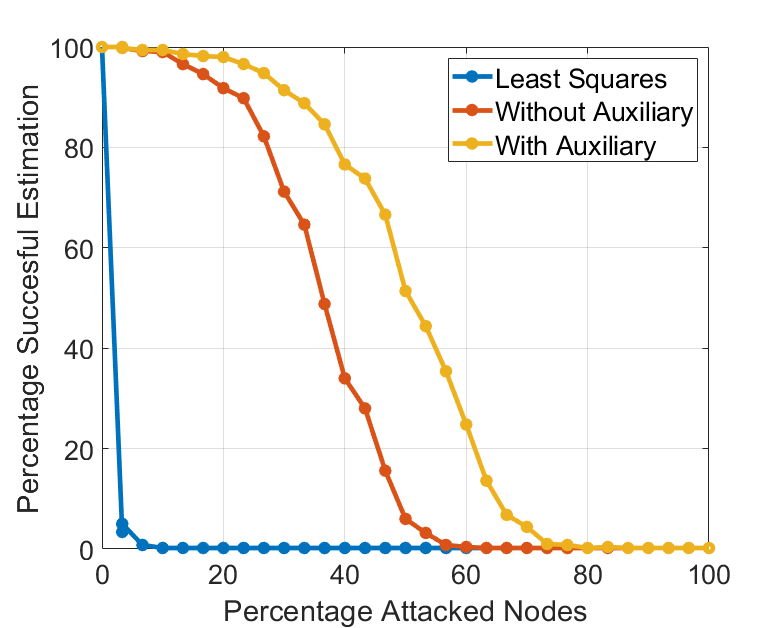}}
    \caption{Simulation results for targeted sensor measurements.\\ Attack vectors are generated to bias select measurements locations by 500\% of its true value along a randomly chosen direction. Plots is the percentage of successful estimations vs. the percentage of attacked sensor nodes.}
    \label{fig:results_target_measurements}
\end{figure}

\begin{figure}[t]
    \centering{\includegraphics[width=3.4in]{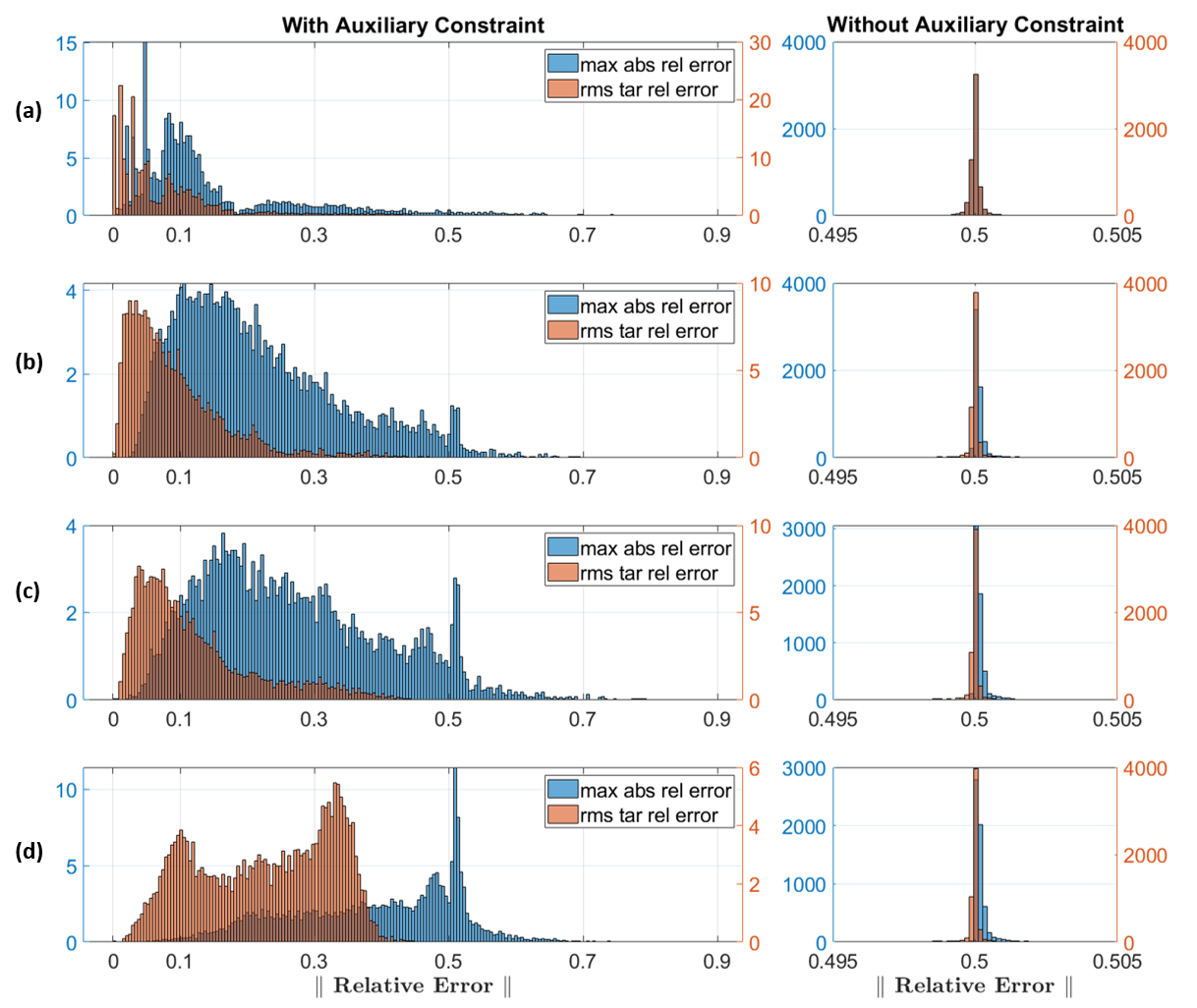}}
    \caption{Simulation results for targeted state FDIA.\\ Attack vectors are generated to bias particular state variables by 50\%. Plotted are the distribution of the \texttt{rms} values of relative errors for targeted states and maximum absolute relative error over all state variables. Subplots: (a) 1 targeted state variable, (b) 5 targeted state variables (c) 10 targeted state variables, (d) 20 targeted state variables.}
    \label{fig:results_target}
\end{figure}

Fig. \ref{fig:results_target_measurements} and Fig. \ref{fig:results_target} show the performance of the proposed algorithm, compared with other standard methods in literature, to the two kinds of FDIA described above. 

For the first set of results, three different state estimation algorithms are simulated against a FDIA directed at specific measurement locations. The three algorithms are: \textit{1)} standard least squares ($\hat{\mathbf{x}} = \argmin\norm{\mathbf{y}-H\mathbf{x}}_2^2$), \textit{2)}  re-weighted $\ell_1$ without the auxiliary model constraint and \textit{3)}  the proposed re-weighted $\ell_1$ with auxiliary model constraint. There are 109 load flow measurements in the simulation. Each simulated scenario, circle points in Fig.~\ref{fig:results_target_measurements}, examines 200 simulations (per state estimation method) with random combinations of sensor locations having fixed percentage (\texttt{x-axis} values) of sensor nodes under attack. 

For the second set of results, the attacks were created in the range space of the system Jacobian matrices. It is well known (e.g., \cite{liu2011false}) that both unconstrained methods will behave similarly under this class of attacks. Thus, we restrict our comparison only to the re-weighted $\ell_1$ algorithms -- one with auxiliary constraints and the other without. Fig.~\ref{fig:results_target} shows the simulations results for four different cases with different numbers of targeted state variables. Fig.~\ref{fig:results_target} shows two plots for each case side-by-side -- one with auxiliary constraints and the other without. Each plot contains the distributions of the maximum absolute relative error, as well as the root-mean-square (\texttt{rms}) values of the relative error for the targeted states. As can be seen from the figures, re-weighted $\ell_1$ algorithms without auxiliary constraints, even though significantly outperforms least-squares based methods in general, are not resilient against state-targeted FDIA.

The proposed re-weighted $\ell_1$ with auxiliary constraints shows significant improvement for both performance indicators. Noticeable effects of the state-targeted FDIA begin to appear when 10 or more states are targeted. This requires compromising more or less 85\% of the system measurement, a feat that demands tremendous amount of resources from any malicious actor. 
\section{Conclusion and future Work}\label{s:conclusions}
In this paper, we showed that incorporation of prior measurement information can significantly improve the resiliency of optimal state estimation algorithms. In particular, we proved that certain prior set inclusion constraints results in much stronger reconstruction error bound. The problem is formulated as a constrained compressive sensing problem and standard results were extended to prove the main results. In addition, numerical simulations were used to validate the theoretical claims by developing a re-weighted $\ell_1$ minimization-base resilient state estimation algorithm for power systems in which data acquired from various IIOT sensors and devices are poisoned with false data injection attacks. The particular case tested is the IEEE 14-bus system mapped to actual NYISO load data. Thus, by corroborating the state estimation with prior auxiliary model, we have demonstrated that it is possible to make it much more difficult to attack a CPS just by corrupting portion of its sensor measurements.

Our future work will aim to extend the theoretical and algorithmic developments in this paper to:
\begin{itemize}
    \item incorporate additional auxiliary information in the estimation, as well as evaluate the developed algorithms through digital real-time simulation platforms using both simulated and field data
    \item the dynamic case using multiple event-triggered auxiliary models
    \item apply the results to the distributed resilient state estimators and moving horizon estimators.
    \item the nonlinear case via infinite-dimensional compressive sensing in Banach space.
\end{itemize}
Moreover, there are interesting theoretical questions that remain open; For instance, what is the resulting stability assessments and margins of the resulting closed loop system when the resilient estimator is used as a dynamic filter, whereby the estimated states are fed into the underlying controller(s)?. An answer to these questions, and likes, will help us judge the quality of an auxiliary model required to achieve a given success rate. Finally, we aim to apply this approach to more examples of CPSs.


%



\section*{Acknowledgment}
Authors are grateful to the Florida State University (FSU) Council on Research and Creativity (CRC) for funding this effort through the First Year Assistant Professor Award Program (FYAP) \#043354

\ifCLASSOPTIONcaptionsoff
  \newpage
\fi



%
\bibliographystyle{IEEEtran}
\bibliography{refs}

%








\end{document}